\documentclass[12pt,amsymb,fullpage]{amsart}
\usepackage{amssymb,amscd,pstricks}

\newtheorem{theorem}{Theorem}[section]
\newtheorem{defn}[theorem]{Definition}

\newtheorem{lemma}[theorem]{Lemma}

\newtheorem{eple}[theorem]{Example}
\newtheorem{rmk}[theorem]{Remarks}
\newtheorem{dsc}[theorem]{Discussion}
\newtheorem{nota}[theorem]{Notation}

\newsavebox{\indbin}
\savebox{\indbin}{\begin{picture}(0,0)
\newlength{\gnu}
\settowidth{\gnu}{$\smile$} \setlength{\unitlength}{.5\gnu}
\put(-1,-.65){$\smile$} \put(-.25,.1){$|$}
\end{picture}}

\newcommand{\be}{\begin{enumerate}}
\newcommand{\bd}{\begin{defn}}
\newcommand{\bt}{\begin{theorem}}
\newcommand{\bl}{\begin{lemma}}
\newcommand{\ee}{\end{enumerate}}
\newcommand{\ed}{\end{defn}}
\newcommand{\et}{\end{theorem}}
\newcommand{\el}{\end{lemma}}

\begin{document}
\title{A Simple Proof of the Fourier Inversion Theorem Using Nonstandard Analysis}
\author{Tristram de Piro}
\address{Mathematics Department, Harrison Building, Streatham Campus, University of Exeter, North Park Road, Exeter, Devon, EX4 4QF, United Kingdom}
 \email{tdpd201@exeter.ac.uk}
\maketitle
\begin{abstract}

We give a proof of the Fourier Inversion Theorem, using the methods of nonstandard analysis.
\end{abstract}

We first make the following, which can be found in \cite{SS};\\

\begin{defn}
\label{schwartz}
We denote by the Schwartz space $S(\mathcal{R})$, the set of all functions $g:\mathcal{R}\rightarrow\mathcal{C}$, such that $g$ and all its derivatives $\{g',g'',\ldots,g^{(n)},\ldots,\}_{n\in\mathcal{N}}$ are rapidly decreasing, in the sense that;\\

$sup_{x\in\mathcal{R}}|x|^{k}|g^{(n)}(x)|<\infty$. (for all $k,n\geq 0$)\\

For such a function $g$, we define its Fourier transform by;\\

$\hat{g}(t)=\int_{-\infty}^{\infty} g(x)e^{-\pi ixt} dx$\\

\end{defn}

\begin{rmk}
\label{remark,schwartz}
It is a well known fact that, if $g\in S(\mathcal{R})$, then its Fourier transform $\hat{g}\in S(\mathcal{R})$ as well, see \cite{SS}. However, this is, perhaps, not the usual definition of the Fourier transform. In \cite{SS}, it is given as;\\

$\hat{g}(t)=\int_{-\infty}^{\infty}g(x)e^{-2\pi ixt} dx$\\

while, in \cite{Kor}, it is defined as;\\

$\hat{g}(t)=\int_{-\infty}^{\infty} g(x)e^{-ixt} dx$\\

Of course, these definitions only differ by a scaling factor, but for each one you choose, you get a distinct rescaled statement of the Inversion Theorem. Once you have proved the Fourier Inversion theorem for one definition, you obtain the other statements by a simple change of variables. The reason for our choice of notation will become apparent later.

\end{rmk}

\begin{theorem}{Fourier Inversion Theorem}\\
\label{fit}

Let $g\in S(\mathcal{R})$, then;\\

$g(x)={1\over 2}\int_{-\infty}^{\infty} \hat{g}(t)e^{\pi ixt}dt$ for all $x\in{\mathcal R}$.

\end{theorem}

\begin{rmk}
\label{remark,fit}
There are many standard proofs of this result, for example in \cite{SS}. This is not the best statement possible. In \cite{Kor}, the requirement that $g\in S(\mathcal{R})$ is weakened to $g\in L^{1}(\mathcal{R})\cap C$ and $\hat{g}\in L^{1}(\mathcal{R})\cap C$, where $C$ denotes the space of complex valued continuous functions on $\mathcal{R}$. In our proof, we do not actually require that $g\in S(\mathcal{R})$, but we do need some assumptions about the differentiability of $g$, and also about its rate of decrease. We have chosen this assumption, mainly because the Schwartz space seems to be often used in the presentation of the Fourier Inversion Theorem.
\end{rmk}

We now introduce the principal spaces which we are going to work with;\\

\begin{defn}
\label{spaces}
Let $\eta\in{{^{*}{\mathcal{N}}}\setminus\mathcal{N}}$, and $\omega\in{^{*}{\mathcal{N}}}$, with $\omega\geq n\eta$, for \emph{all} $n\in\mathcal{N}$. We define;\\

$\overline{\mathcal R}_{\omega,\eta}=\{\tau\in{^{*}{\mathcal R}}:-{\omega\over\eta}\leq\tau<{\omega\over\eta}\}$\\

We let $\mathfrak{C}$ be the ${*}$-finite algebra consisting of internal unions of intervals of the form $[{i\over\eta},{i+1\over\eta})$, for $-\omega\leq i<\omega$.\\

We define a counting measure on $\mathfrak{C}$ by $\lambda([{i\over\eta},{i+1\over\eta}))={1\over\eta}$.\\

Then $(\overline{\mathcal R}_{\omega,\eta},\mathfrak{C},\lambda)$ is a hyperfinite measure space with $\lambda(\overline{\mathcal R}_{\omega,\eta})={2\omega\over\eta}$.\\

We denote by $(\overline{\mathcal R}_{\omega,\eta},L(\mathfrak{C}),L(\lambda))$ the associated Loeb space, (\footnote{\label{henson} The existence of such a space follows from \cite{Loeb}. However, the uniqueness of the extension of ${^{\circ}\lambda}$ to $\sigma(\mathfrak{C})$ was only shown there in the case that $\lambda$ is finite. Later, Ward Henson proved the uniqueness of the extended measure, even in the case that $\lambda$ is infinite. After producing the extension, we are then passing to the completion, see \cite{cut}.}).\\

We let $(\mathcal{R},\mathfrak{B},\mu)$ denote the completion of the Borel field $\mathfrak{D}$ on $\mathcal{R}$, with respect to Lebesgue measure $\mu$, (\footnote{\label{uniquelebesgue} Again, Caratheodory's Theorem provides the existence of Lebesgue measure $\mu$ on the $\sigma$-algebra $\mathfrak{D}$ generated by the open sets. Uniqueness of the extension follows easily by restricting to finite intervals.}).\\

We let $\mathcal{R}^{+-\infty}$ denoted the extended real line $\mathcal{R}\cup\{+\infty,-\infty\}$, and let $\{g_{\infty},\hat{g}_{\infty}\}$ be the extensions of functions in Definition \ref{schwartz}, obtained by setting $g_{\infty}(+\infty)=g_{\infty}(-\infty)=0$, and similarly for $\hat{g}_{\infty}$.

\end{defn}

\begin{lemma}
\label{unique}
There exists a unique $\sigma$-algebra $\mathfrak{B}'$ on $\mathcal{R}^{+-\infty}$, which separates the points $+\infty$ and $-\infty$, and such that $\mathfrak{B}'|_{\mathcal{R}}=\mathfrak{B}$. Moreover, there is a unique extension of $\mu$ to a measure $\mu'$ on $\mathfrak{B}'$ with the property that $\mu(\infty)=\mu(-\infty)=\infty$. The same holds with $\mathfrak{D}$ and $\mathfrak{D}'$ replacing $\mathfrak{B}$ and $\mathfrak{B}'$. The resulting measure space $(\mathcal{R}^{+-\infty},\mathfrak{B}',\mu')$ is the completion of $(\mathcal{R}^{+-\infty},\mathfrak{D}',\mu')$.

\end{lemma}

\begin{proof}
The construction of $\mathfrak{B}'$ is easy. We let $\mathfrak{B}_{+\infty}$ consist of all sets of the form $B\cup\{+\infty\}$, where $B\in\mathfrak{B}$, and, similarly, define $\mathfrak{B}_{-\infty}$ and $\mathfrak{B}_{+-\infty}$. Then, let $\mathfrak{B}'=\mathfrak{B}\cup\mathfrak{B}_{+\infty}\cup\mathfrak{B}_{-\infty}\cup\mathfrak{B}_{+-\infty}$. Clearly, $\mathfrak{B}'$ separates the points $+\infty$ and $-\infty$, moreover $\mathfrak{B}'|_{\mathcal{R}}=\mathfrak{B}$. It is a simple exercise to verify that $\mathfrak{B}$ is a $\sigma$-algebra. In order to see uniqueness, let $\mathfrak{B}''$ have these properties. As $\mathfrak{B}''|_{\mathcal{R}}=\mathfrak{B}$, we have $\mathfrak{B}\subset\mathfrak{B}''$. Choose a set $B$ containing $+\infty$, but not $-\infty$, then $\{+\infty\}=B\cap\bigcap_{n\in\mathcal{N}}(-n,n)^{c}$ belongs to $\mathfrak{B}''$. Moreover $\{+\infty,-\infty\}={\mathcal{R}^{+-\infty}\setminus\mathcal{R}}$ belongs to $\mathfrak{B}''$, so, $-\infty$ belongs to $\mathfrak{B}''$. Hence, $\mathfrak{B'}\subset\mathfrak{B}''$. If $C$ belongs to $\mathfrak{B}''$, then clearly $C\cap\mathcal{R}\in\mathfrak{B}$, so it must be of the above form, that is $\mathfrak{B}'=\mathfrak{B}''$. Now define $\mu'$ by setting $\mu=\mu'$ on $\mathfrak{B}$, and letting $\mu'(C)=\infty$, for any $C\in{\mathfrak{B}'\setminus\mathfrak{B}}$. It is straightforward to see that $\mu'$ defines a measure, with $\mu'(\infty)=\mu'(-\infty)=\infty$, extending $\mu$. If $\mu''$ satisfies these properties, then as any set $C\in{\mathfrak{B}'\setminus\mathfrak{B}}$ contains at least one of $\{+\infty,-\infty\}$, it must be $\infty$ on these sets, so $\mu'=\mu''$. Exactly the same argument gives the result for $\mathfrak{D}$ and $\mathfrak{D}'$. The completeness statement follows directly as $(\mathcal{R},\mathfrak{B},\mu)$ is complete, and any set of measure $0$, $\mu'$, in $\mathfrak{B}'$, belongs to $\mathfrak{B}$.

\end{proof}

\begin{theorem}
\label{mmp}
The standard part mapping;\\

$st:(\overline{\mathcal R}_{\omega,\eta},L(\mathfrak{C}),L(\lambda))\rightarrow(\mathcal{R}^{+-\infty},\mathfrak{B}',\mu')$\\

is measurable and measure preserving. In particular, if $\{g_{\infty},\hat{g}_{\infty}\}$ are as in Definition \ref{spaces}, and $\{st^{*}(g_{\infty}),st^{*}(\hat{g}_{\infty})\}$ are their pullbacks under $st$, then, $\{st^{*}(g_{\infty}),st^{*}(\hat{g}_{\infty})\}$ are integrable with respect to $L(\lambda)$, $\{g_{\infty},\hat{g}_{\infty}\}$ are integrable with respect to $\mu'$, $\{g,\hat{g}\}$ are integrable with respect to $\mu$, and;\\

$\int_{\overline{\mathcal R}_{\omega,\eta}}st^{*}(g_{\infty})dL(\lambda)=\int_{\mathcal{R}^{+-\infty}} g_{\infty}d \mu'=\int_{\mathcal{R}}g d\mu$\\

$\int_{\overline{\mathcal R}_{\omega,\eta}}st^{*}(\hat{g}_{\infty}) dL(\lambda)=\int_{\mathcal{R}^{+-\infty}} \hat{g}_{\infty} d \mu'=\int_{\mathcal{R}}\hat{g} d\mu$\\

\end{theorem}

\begin{proof}
We let $\Sigma_{0}'\subset\mathfrak{B}'$ denote the sets consisting of finite unions of the form;\\

$[-\infty,b_{1})\cup[a_{2},b_{2})\cup\ldots\cup[a_{r},b_{r})\cup [b_{r+1},\infty]$\\

where $b_{1}\leq a_{2}\ldots\leq b_{r+1}$ belong to $\mathcal{R}$. It is an easy exercise to check that $\Sigma_{0}$ is an algebra. Let $\mathfrak{D}'\subset\mathfrak{B}'$ be the $\sigma$-algebra generated by $\Sigma_{0}'$. Then $\mathfrak{D}'|_{\mathcal{R}}$ is just the Borel field $\mathfrak{D}$ on $\mathcal{R}$, and by Lemma \ref{unique}, $\mathfrak{D}'$ is obtained from $\mathfrak{D}$ by adjoining at least one of the points $\{+\infty,-\infty\}$. Then $\mathfrak{B}'$ is just the completion of $\mathfrak{D}'$ with respect to $\mu'|_{\mathfrak{D}'}$, using the definition of $\mathfrak{B}$ and the fact that $\mathfrak{B}'|_{\mathcal{R}}=\mathfrak{B}$. Now, if $a,b\in{\mathcal{R}}$;\\

$st^{-1}([a,b))=\bigcup_{m=1}^{\infty}\bigcap_{n=1}^{\infty}[{[\eta(a-{1\over n})]\over\eta},{[\eta(b-{1\over m})]\over\eta})$\\

$st^{-1}([-\infty,a))=\bigcup_{m=1}^{\infty}[{-\omega\over\eta},{[\eta(a-{1\over m})]\over\eta})$ $(*)$\\

where $[\ ]$ denotes integer part. Observing that $\{i\in{^{*}{\mathcal Z}}:-\omega\leq i\leq [\eta(b-{1\over m})]-1\}$ is internal, these sets belong to $L(\mathfrak{C})$. Now consider $\{B\in\mathfrak{B}':st^{-1}(B)\in L(\mathfrak{C})\}$. This is a $\sigma$-algebra containing $\mathfrak{D}'$ by $(*)$. In particular, $st^{-1}(-\infty)$ and $st^{-1}(+\infty)$ belong to $L(\mathfrak{C})$. Moreover;\\

$L(\lambda)(st^{-1}([a,b)))=lim_{m\rightarrow\infty}lim_{n\rightarrow\infty}{^{\circ}(b-a+{1\over n}-{1\over m})}=(b-a)$\\

$L(\lambda)(st^{-1}(+\infty))=L(\lambda)(st^{-1}(-\infty))=\infty$ $(**)$\\

In the first claim, we have used elementary properties of measures on $\sigma$-algebras and the definition of $\lambda|\mathfrak{C}$. In the second claim, we have used the fact that $st^{-1}(+\infty)\supset [{\omega\over 2\eta},{\omega\over\eta})$, and $L(\lambda)([{\omega\over 2\eta},{\omega\over\eta}))={^{\circ}({\omega\over 2\eta})}=\infty$, by the choice of $\omega$. Similarly, for $st^{-1}(-\infty)$. It follows that the push forward measure $st_{*}(L(\lambda))$ on $\mathfrak{D}'$, agrees with $\mu$ on the algebra $\Sigma_{0}|\mathcal{R}$, hence, by footnote \ref{uniquelebesgue}, it agrees with $\mu$ on $\mathfrak{D}=\mathfrak{D'}|_{\mathcal{R}}$. By Lemma \ref{unique}, it agrees with $\mu'$ on $\mathfrak{D}'$. Now if $B\in\mathfrak{B}'$, we can find $C\subset B\subset D$, with $C$ and $D$ belonging to $\mathfrak{D}'$, such that $\mu'({D\setminus C})=0$. Then $st^{-1}(C)\subset st^{-1}(B)\subset st^{-1}(D)$ and $L(\lambda)(st^{-1}({D\setminus C})=L(\lambda)({st^{-1}(D)\setminus st^{-1}(C)})=0$. Hence, as $(\overline{\mathcal R}_{\omega,\eta},L(\mathfrak{C}),L(\lambda))$ is complete, we have that $st^{-1}(B)\in L(\mathfrak{C})$ and $L(\lambda)(st^{-1}(B))=L(\lambda)(st^{-1}C)=\mu'(C)=\mu'(B)$, as required. For the second part of the theorem, observe that $S(\mathcal{R})\subset L^{1}(\mathcal{R})$ and use Remarks \ref{remark,schwartz}. Clearly, the extensions $\{g_{\infty},\hat{g}_{\infty}\}$ are $\mathfrak{D}'$-measurable. Using \cite{Rud}(Definition 1.23), and Lemma \ref{extension};\\

$\int_{\mathcal{R}^{+-\infty}} g_{\infty}d \mu'=\int_{\mathcal{R}} g_{\infty}d \mu' +\int_{\{+\infty,-\infty\}} g_{\infty}d \mu'=\int_{\mathcal{R}} g d \mu$\\

and, similarly, for $\hat{g}$. Then, it follows, using the first part of the Theorem, and Lemma \ref{cov}, that, $\{st^{*}(g_{\infty}),st^{*}(\hat{g}_{\infty})\}$ are integrable with respect to $L(\lambda)$, and;\\

$\int_{\overline{\mathcal R}_{\omega,\eta}}st^{*}(g_{\infty})dL(\lambda)=\int_{\mathcal{R}^{+-\infty}} g_{\infty}d \mu'$\\

and, similarly, for $st^{*}(\hat{g}_{\infty})$.

 \end{proof}

We make the following;\\

\begin{defn}
\label{character}

Let $(G,+,0)$ be a finite commutative group, and let $(\mathcal{C}^{*},\centerdot,1)$ denote the multiplicative group of complex numbers, with absolute value $1$, then by a a character $\gamma$ of $G$, we mean a homomorphism $\gamma:G\rightarrow\mathcal{C}^{*}$.\\

Let $m,n\in{\mathcal N}_{>0}$. We let $(\mathcal{Z}_{m},+,0)=({\mathcal{Z}/{m\mathcal{Z}}},+,0)$ denote the additive group of integers $mod$ $m$. For $x,y\in\mathcal{Z}_{m}$, we let $xy$ denote ordinary multiplication in $\mathcal{Z}$, where $\{x,y\}$ are uniquely represented in $\{0,\ldots,m-1\}$\\

$G_{2m}=\{-m,-(m-1),\ldots,m-1\}$ denotes the group of order $2m$, with addition given by $m_{1}+m_{2}=S^{m_{1}}(m_{2})$, where $S$ is the shift map $S(x)=x+1$ if $x\neq m-1$, $S(m-1)=-m$.\\

$G_{m,n}=\{{-m\over n},{-(m-1)\over n},\ldots,{m-1\over n}\}$ denotes the group of order $2m$, with addition as defined for $G_{m}$. As before, for $x\in G_{m,n}$,$y\in G_{m,n}$ or $y\in{\mathcal Z}$,  we let $xy$ denote ordinary multiplication in $\mathcal{Z}$. \\

For a finite commutative group $G$, we let $\mathfrak{G}$ denote the finite $\sigma$-algebra consisting of all subsets of $G$, and $\mu_{G}$ the associated probability measure. $L^{1}(G)$ denotes the set of functions $g:G\rightarrow\mathcal{C}$. For $g,h\in L^{1}(G)$, we let $<g,h>=\int_{G} g{\bar h} d\mu_{G}$.

\end{defn}

The following can be found in \cite{lux};\\

\begin{theorem}
\label{charbasis}
Let $(G,+,0)$ be a finite commutative group of order $m$, then there exist exactly $m$ characters on $G$, and they form an orthonormal basis of $L^{1}(G)$, with respect to $<,>$, (\footnote{\label{orthbasis}It is shown in \cite{lux} that the characters form an orthogonal basis with respect to the measure $m\mu_{G}$. However, it is then a simple computation, using the definition of a character, to see that they are an orthonormal basis with respect to the probability measure $\mu_{G}$}). The characters on $\mathcal{Z}_{m}$ are given by;\\

$\gamma_{k}(x)=exp({2\pi i\over m} kx)$ for $k\in\{0,1,\ldots,m-1\}$\\

\end{theorem}

\begin{defn}
\label{chartransf}
Let $(G,+,0)$ be a finite commutative group of order $m$, and let $G_{*}$ denote its commutative group of characters, of order $m$, (\footnote{\label{isomo} In fact, $G$ and $G_{*}$ are isomorphic, see \cite{lux}.}), then, if $g\in L^{1}(G)$, we define $\hat{g}:G_{*}\rightarrow\mathcal{C}$, by;\\

$\hat{g}(\gamma)=<g,\gamma>=\int_{G} g{\bar\gamma} d\mu_{G}$\\

\end{defn}

We then obtain;\\

\begin{theorem}{Inversion Theorem for Finite Groups}\\
\label{fgit}

Let $\{G,G_{*},g,\hat{g}\}$ be as in Definition \ref{chartransf}, then;\\

$g(x)=\sum_{j=0}^{m-1}\hat{g}(\gamma_{j})\gamma_{j}(x)$\\

where $x\in G$, and $j$ enumerates $G_{*}$.

\end{theorem}

\begin{proof}
This is almost immediate. By Theorem \ref{charbasis};\\

$g=\sum_{j=0}^{m-1}<g,\gamma_{j}>\gamma_{j}$ in $L^{1}(G)$\\

Then, by Definition \ref{chartransf}, and the fact that $\mu_{G}(x)>0$, if $x\in G$;\\

$g(x)=\sum_{j=0}^{m-1}<g,\gamma_{j}>\gamma_{j}(x)=\sum_{j=0}^{m-1}\hat{g}(\gamma_{j})\gamma_{j}(x)$\\

\end{proof}

We now compute the character group on $G_{m,n}$;\\

\begin{lemma}
\label{partichar}
Let $G_{m,n}$ be as in Definition \ref{character}, then the characters on $G_{m,n}$ are given by;\\

$\gamma_{y}(x)=exp({\pi i n^{2}\over m} xy)$\\

where $x,y\in G_{m,n}$.

\end{lemma}

\begin{proof}
First observe that there exists an isomorphism $\phi:G_{m}\rightarrow\mathcal{Z}_{2m}$, defined by $\phi(x)=(x+2m)_{mod 2m}$. Hence, by Theorem \ref{charbasis}, the characters on $G_{m}$ are given by;\\

$exp({2\pi i\over 2m}(x+2m)_{mod 2m}j)=exp({\pi i\over m}(x+2m)_{mod 2m}j)=exp({\pi i\over m} xj)$\\

where $x\in G_{m}$, $j\in\{0,1,\ldots,2m-1\}$. Here, we have also used the facts that;\\

${[x+2m]_{mod 2m}\over m}={x\over m}$, if $0\leq x\leq m-1$\\

${[x+2m]_{mod 2m}\over m}={x\over m}+2$, if $-m\leq x<0$\\

and $exp(2\pi i)=1$. Now writing $j=y+m$, for $y\in G_{m}$, we obtain that;\\

$exp({\pi i\over m} xj)=-exp({\pi i\over m}y)=exp({\pi i\over m}(y-m))$\\

Observe that the characters $exp({\pi i\over m}(y-m))$ correspond to $exp({\pi i\over m}y')$, where $y'=y-m$ belongs to $\{-m,\ldots,-1\}$ if $y\in\{0,\ldots,m-1\}$, and correspond to $exp({\pi i\over m}y'')$, where $y''=y+m$ belongs to $\{0,\ldots,m-1\}$ if $y\in\{-m,\ldots,-1\}$. Hence, the characters in ${G_{m}}_{*}$ are given by;\\

$\gamma_{y}(x)=exp({\pi i\over m} xy)$ $(*)$\\

for $x,y\in G_{m}$. Now observe there exists an isomorphism $\psi:G_{m,n}\rightarrow G_{m}$ defined by $\psi(x)=nx$. Hence, by $(*)$, the characters in ${G_{m,n}}_{*}$ are given by;\\

$\gamma_{y}(x)=exp({\pi i\over m} (nx)(ny))=exp({\pi i n^{2}\over m}xy)$\\

for $x,y\in G_{m,n}$.
\end{proof}

\begin{defn}
\label{partictransf}

Let $n\in{\mathcal N}_{>0}$, let $G_{n^{2},n}$ be the group of order $2n^{2}$, as in Definition \ref{character}, and let $g\in L^{1}(G_{n^{2},n})$. Let $\mathfrak{G}$ be as before, and let $\lambda_{G}$ be the rescaled measure, given by $\lambda_{G}=2n\mu_{G}$. Then, we define $\hat{g}\in L^{1}(G_{n^{2},n})$ to be the function;\\

$\hat{g}(t)=\int_{G_{n^{2},n}} g(x)exp(-\pi ixt) d\lambda_{G}$ $(t\in G_{n^{2},n},x\in G_{n^{2},n})$\\

\end{defn}

\begin{theorem}{Inversion Theorem for $G_{n^{2},n}$}\\
\label{pgit}

Let $\{G_{n^{2},n},\lambda_{G},g,\hat{g}\}$ be as in Definition \ref{partictransf}, then;\\

$g(x)={1\over 2}\int_{G_{n^{2},n}}\hat{g}(t)exp(\pi ixt) d\lambda_{G}$ $(x\in G_{n^{2},n})$\\

\end{theorem}

\begin{proof}

By Lemma \ref{partichar}, the characters on $G_{n^{2},n}$ are given by;\\

$\gamma_{y}(x)=exp({\pi i n^{2}\over n^{2}}xy)=exp(\pi ixy)$ $(*)$\\

for $x,y\in G_{n^{2},n}$. Using Definition \ref{chartransf}, and the fact that $\mu_{G}(x)={1\over 2n^{2}}$, for $x\in G_{n^{2},n}$, we have;\\

$\hat{g}(\gamma_{y})={1\over 2 n^{2}}\sum_{k=-n^{2}}^{n^{2}-1}g({k\over n})exp(-\pi i{k\over n}y)$ $(**)$\\

where $y\in G_{n^{2},n}$. By Theorem \ref{fgit}, $(*),(**)$ and the fact that $\lambda_{G}(x)={1\over n}$, for $x\in G_{n^{2},n}$;\\

$g(x)=\sum_{l=-n^{2}}^{n^{2}-1}\hat{g}(\gamma_{{l\over n}})\gamma_{l\over n}(x)$\\

$=\sum_{l=-n^{2}}^{n^{2}-1}\hat{g}(\gamma_{{l\over n}})exp(\pi i{lx\over n})$\\

$=\sum_{l=-n^{2}}^{n^{2}-1}[{1\over 2 n^{2}}\sum_{k=-n^{2}}^{n^{2}-1}g({k\over n})exp(-\pi i{k\over n}{l\over n})]exp(\pi i{lx\over n})$\\

$={1\over 2}{1\over n}\sum_{l=-n^{2}}^{n^{2}-1}[{1\over n}\sum_{k=-n^{2}}^{n^{2}-1}g({k\over n})exp(-\pi i{k\over n}{l\over n})]exp(\pi i{lx\over n})$\\

$={1\over 2}{1\over n}\sum_{l=-n^{2}}^{n^{2}-1}[\int_{G_{n^{2},n}} g(y)exp(-\pi iy{l\over n})d\lambda_{G}]exp(\pi i{lx\over n})$\\

$={1\over 2}{1\over n}\sum_{l=-n^{2}}^{n^{2}-1}\hat{g}({l\over n})exp(\pi i{lx\over n})$\\

$={1\over 2}\int_{G_{n^{2},n}} \hat{g}(t)exp(\pi ixt)d \lambda_{G}$\\

\end{proof}
\begin{defn}
\label{transfdefs}
We let $\overline{{\mathcal R}_{\eta}}={\overline{\mathcal{R}_{{\eta}^{2},\eta}}}$ and let $\{\mathfrak{C}_{\eta},\lambda_{\eta}\}$ be as before. We let $\mathfrak{C}_{\eta}^{2}$ denote the ${^{*}}$-finite algebra on ${\overline{{\mathcal R}_{\eta}}}^{2}$, consisting of internal unions of the form $[{k\over\eta},{k+1\over\eta})\times[{j\over\eta},{j+1\over\eta})$, $-\eta^{2}\leq k,j<\eta^{2}$, and $\lambda_{\eta}^{2}$ be the counting measure on $\mathfrak{C}_{\eta}^{2}$, defined by $\lambda_{\eta}^{2}([{k\over\eta},{k+1\over\eta})\times[{j\over\eta},{j+1\over\eta}))={1\over{\eta^{2}}}$.\\

We let ${^{*}exp(\pi i xt)},{^{*}exp(-\pi i xt)}:{^{*}{\mathcal{R}}}^{2}\rightarrow{^{*}{\mathcal C}}$ be the transfers of the functions $exp(\pi ixt),exp(-\pi ixt):{\mathcal R}^{2}\rightarrow{\mathcal C}$, and use the same notation to denote the restrictions of the transfers to ${\overline{{\mathcal R}_{\eta}}}^{2}$.\\

We let $exp_{\eta}(\pi ixt),exp_{\eta}(-\pi ixt):{\overline{{\mathcal R}_{\eta}}}^{2}\rightarrow{^{*}{\mathcal C}}$ denote their $\mathfrak{C}_{\eta}^{2}$-measurable counterparts, defined by;\\

$exp_{\eta}(\pi ixt)={^{*}exp(\pi i{[\eta x]\over\eta}{[\eta t]\over\eta})}$, $(x,t)\in{\overline{{\mathcal R}_{\eta}}}^{2}$\\

and, similarly, for $exp_{\eta}(-\pi ixt)$. Given $f:\overline{{\mathcal R}_{\eta}}\rightarrow{^{*}\mathcal{C}}$, which is $\mathfrak{C}_{\eta}$-measurable, we define;\\

$\hat{f}_{\eta}(t)=\int_{\overline{{\mathcal R}_{\eta}}}f(x)exp_{\eta}(-\pi ixt) d\lambda_{\eta}$\\

so $\hat{f}_{\eta}:\overline{{\mathcal R}_{\eta}}\rightarrow{^{*}\mathcal{C}}$ is $\mathfrak{C}_{\eta}$-measurable. $(*)$\\

Given $g:{\mathcal R}\rightarrow{\mathcal C}$, we let ${^{*}g}:{^{*}{\mathcal{R}}}\rightarrow{^{*}{\mathcal{C}}}$ denote its transfer and its restriction to $\overline{{\mathcal R}_{\eta}}$. We let $g_{\eta}$ denote its $\mathfrak{C}_{\eta}$-measurable counterpart, as above, and let $\hat{g}_{\eta}$ be as in $(*)$.\\

For $n\in\mathcal{N}$, we let ${\mathcal R}_{n}=\overline{{\mathcal R}_{n}}\cap\mathcal{R}$. We let $\mathfrak{C}_{n,st}$ consist of all finite unions of intervals of the form $[{i\over n},{i+1\over n})$, for $-n^{2}\leq i\leq n^{2}-1$. $\lambda_{n,st}$ is defined on $\mathfrak{C}_{n,st}$, by setting $\lambda_{n}([{i\over n},{i+1\over n}))={1\over n}$.\\

 $\{\mathfrak{C}_{n,st}^{2},\lambda_{n,st}^{2}, exp_{n,st}(\pi i xt), exp_{n,st}(-\pi i xt)\}$ are all defined as above, restricting to $\mathcal{R}$. If $g:{\mathcal R}\rightarrow{\mathcal C}$, we similarly define, $\{g_{n,st},\hat{g}_{n,st}\}$, ($st$ is suggestive notation for standard). Observe that $\lambda_{n,st}$ is just the restriction of Lebesgue measure $\mu$ to $\mathfrak{C}_{n,st}$, and transfers to $\lambda_{n}$.\\

$\{exp_{n,st}(\pi i xt), exp_{n,st}(-\pi i xt),g_{n,st},\hat{g}_{n,st}\}$ are all standard functions, which transfer to $\{exp_{n}(\pi i xt), exp_{n}(-\pi i xt),g_{n},\hat{g}_{n}\}$.\\

Finally, we let $\mathfrak{C}_{n,ext}$ denote the $\sigma$-algebra on $\mathcal{R}$, consisting of countable unions of intervals of the form $[{i\over n},{i+1\over n})$, for $i\in{\mathcal Z}$, and $\lambda_{n,ext}$ be the corresponding measure. We similarly define $\{\mathfrak{C}_{n,ext}^{2},\lambda_{n,ext}^{2},exp_{n,ext}(\pi i xt),\\
 exp_{n,ext}(-\pi i xt)\}$\\

If $g:{\mathcal R}\rightarrow{\mathcal C}$, we let $g_{n,ext}:{\mathcal R}\rightarrow{\mathcal C}$ be the $\mathfrak{C}_{n,ext}$-measurable function obtained by setting $g_{n,ext}(x)=g({[nx]\over n})$, so $g_{n,ext}|_{\mathcal{R}_{n}}=g_{n,st}$.

\end{defn}

We now have;\\

\begin{lemma}{Inversion Theorem for $\overline{{\mathcal R}_{\eta}}$}\\
\label{nsit}

Let $\{\overline{{\mathcal R}_{\eta}},\lambda_{\eta},f,\hat{f}_{\eta}\}$ be as in Definition \ref{transfdefs}, then;\\

$f(x)={1\over 2}\int_{\overline{{\mathcal R}_{\eta}}}\hat{f}_{\eta}(t)exp_{\eta}(\pi ixt) d\lambda_{\eta}(t)$ $(x\in\overline{{\mathcal R}_{\eta}})$\\

\end{lemma}

\begin{proof}
As $f(x)$ is $\mathfrak{C}_{\eta}$-measurable and $exp_{\eta}(\pi ixt)$ is $\mathfrak{C}_{\eta}^{2}$-measurable, both sides of the equation are unchanged if we replace $x$ by ${[\eta x]\over\eta}$. Now the result follows directly, by transfer, from the corresponding result for $G_{n^{2},n}$, Theorem \ref{pgit}, and the definition of the internal integral
$\int_{\overline{{\mathcal R}_{\eta}}}$ on $\overline{{\mathcal R}_{\eta}}$, see Definition \ref{internint},(\footnote{\label{anxious} If the reader is anxious about some ambiguity in transferring double sums or integrals, the important point to realise is that the ${^{*}}$ operator factors through any set of standard predicates or functions, so ${^{*}{\mathcal R}}\models(\forall n\in{^{*}{\mathcal{N}}})({^{*}(S_{1,n}\circ S_{2,n})}=({^{*}S_{1,n}}\circ {^{*}S_{2,n}}))$ if $\{{^{*}S_{1,n}},{^{*}S_{2,n}}\}$ are hyperfinite sums.}).
\end{proof}

We now want to specialise the result of Lemma \ref{nsit} to $(\overline{{\mathcal R}_{\eta}},L(\mathfrak{C}_{\eta}),L(\lambda_{\eta}))$, using Loeb integration theory. The problem now is to obtain the $S$-integrability conditions, see \cite{cut} for a definition of $S$-integrability.

\begin{theorem}
\label{gSinteg}

Let $g\in S({\mathcal R})$, then $g_{\eta}$, as given in Definition \ref{transfdefs}, is $S$-integrable on $\overline{{\mathcal R}_{\eta}}$. Moreover ${^{\circ}{g}}_{\eta}=st^{*}(g_{\infty})$, everywhere $L(\lambda_{\eta})$, and;\\

${^{\circ}\int_{\overline{{\mathcal R}_{\eta}}}g_{\eta} d\lambda_{\eta}}=\int_{\overline{{\mathcal R}_{\eta}}}st^{*}(g_{\infty}) d L(\lambda_{\eta})=\int_{\mathcal{R}} g d\mu$\\
\end{theorem}
\begin{proof}
We first claim that $g_{n,ext}$ is integrable $\mu$, and $lim_{n\rightarrow\infty}||g-g_{n,ext}||_{L^{1}}=0$, $(*)$.
In order to see this, let $\epsilon>0$ be standard, and choose $N\in{\mathcal N}\geq 2$, such that;\\

$\int_{-\infty}^{\infty}|g| d\mu -\int_{-N}^{N} |g| d\mu<{\epsilon\over 3}$\\

and $N>{9C\over\epsilon}$. As $|g|$ is continuous on the interval $[-N,N]$, by Darboux's theorem, see \cite{bs}, there exists $M\in{\mathcal N}$, such that for all $n\geq M$;\\

$\int_{-N}^{N}(|g-g_{n,ext}|) d\mu < {\epsilon\over 3}$\\

Now, for $n\in{\mathcal N}_{>0}$, using Definition \ref{schwartz};\\

$\int_{|x|>N} |g_{n,ext}|(x) d\mu(x)$\\

$={1\over n}(\sum_{|j|\geq Nn+1}|g({j\over n})|+|g(N)|)$\\

$\leq |{g(N)\over n}| +{1\over n}\sum_{|j|\geq Nn+1} {Cn^{2}\over {j}^{2}}$\\

$\leq {C\over N}+Cn\int_{|x|>Nn}{1\over x^{2}}dx$\\

$={C\over N}+{2Cn\over Nn}={3C\over N}<{\epsilon\over 3}$\\

Combining these estimates, it follows that, $g_{n,ext}$ is integrable $\mu$, and for $n\geq M$;\\

$\int_{-\infty}^{\infty}|g-g_{n,ext}| d\mu < \epsilon$\\

As $\epsilon$ was arbitrary, we obtain the result $(*)$. Now, using $(*)$, choose $N_{1}\in\mathcal{N}$, such that $||g\chi_{[L,N]}||_{L^{1}}<{\epsilon\over 2}$ and $||g-g_{n,ext}||_{L^{1}}<{\epsilon\over 2}$, for all $n\in{\mathcal N}_{>0}$, and $L,N\in\mathcal{Z}$, $LN\geq 0$,with $min(n,|L|,|N|)>N_{1}$. Then;\\

$||g_{n,ext}\chi_{[L,N]}||_{L^{1}}\leq ||(g_{n,ext}-g)\chi_{[L,N]}||_{L^{2}}+||g\chi_{L,N}||_{L^{1}}<\epsilon$ $(**)$\\

for all such $\{n,L,N\}$. We now transfer the result $(**)$. We have that;\\

$\mathcal{R}\models (\forall n_{(n>N_{1})})(\forall L,N_{(LN\geq 0, N_{1}<|L|,|N|<n)})\int_{L}^{N}|g_{n,st}| d\lambda_{n,st} <\epsilon$\\

Hence, the corresponding statement is true in ${^{*}\mathcal{R}}$. In particular, if $\eta$ is infinite, and $\{L,N\}$ are infinite, of the same sign, belonging to $\overline{\mathcal{R}}_{\eta}$, we have that;\\

$\int_{L}^{N}|g_{\eta}| d\lambda_{\eta} < \epsilon$\\

As $\epsilon$ was arbitrary we conclude that;\\

$\int_{L}^{N}|g_{\eta}| d\lambda_{\eta}\simeq 0$ $(***)$\\

for all infinite $\{L,N\}$, of the same sign, in $\overline{\mathcal{R}}_{\eta}$. Now consider the internal sequence;\\

$\{s_{n}\}_{1\leq n\leq\eta}=\{\int_{\overline{\mathcal{R}}_{\eta}}(|g_{\eta}-g_{\eta}\chi_{[-n,n)}|) d\lambda_{\eta}\}_{1\leq n\leq\eta}$\\

Then, by $(***)$, $s_{\omega'}\simeq 0$, for all infinite $\omega'\leq\eta$. Applying Theorem \ref{intlim}, we have that $lim_{n\rightarrow\infty}({^{\circ}s_{n}})=0$. That is;\\

$lim_{n\rightarrow\infty}{^{\circ}(\int_{\overline{\mathcal{R}}_{\eta}}|g_{\eta}-g_{\eta}\chi_{[-n,n)}| d\lambda_{\eta})}=0$ $(\dag)$\\

As $g$ is bounded by $M$, the same is true for $g_{\eta}$, hence, the functions $\{g_{\eta}\chi_{[-n,n)}\}$ are finite, in the sense of Definition \ref{finite}. Applying Theorem \ref{andcrit} and $(\dag)$, we obtain that $g_{\eta}$ is $S$-integrable. Now, using the fact that $lim_{x\rightarrow\infty}g(x)=0$, it is straightforward, using Theorem \ref{stlim}, to show that $g_{\eta}(x)\simeq 0$, for all infinite $x\in{\overline{\mathcal{R}}_{\eta}}$. As $g$ is continuous, by Theorem \ref{nscont}, we have that $g_{\eta}(x)={^{*}g}({[\eta x]\over\eta})\simeq g({^{\circ}x})$, for all finite $x\in{\overline{\mathcal{R}}_{\eta}}$. Hence, for all $x\in{\overline{\mathcal{R}}_{\eta}}$, ${^{\circ}g_{\eta}}(x)=st^{*}(g_{\infty})(x)$. Finally, by Theorem \ref{sintegspec} and Theorem \ref{mmp};\\

${^{\circ}\int_{\overline{{\mathcal R}_{\eta}}}g_{\eta} d\lambda_{\eta}}=\int_{\overline{{\mathcal R}_{\eta}}}st^{*}(g_{\infty}) d L(\lambda_{\eta})=\int_{\mathcal{R}} g d\mu$\\
\end{proof}

The corresponding result for $\hat{g}_{\eta}$ is more difficult to show. We require the following;\\

\begin{defn}
\label{discretederivshift}
If $n\in\mathcal{N}$, and $g_{n,st}$ is $\mathfrak{C}_{n,st}$-measurable, we define the discrete derivative $g'_{n,st}$ by;\\

$g'_{n,st}({j\over n})=n(g_{n,st}({j+1\over n})-g_{n,st}({j\over n}))$ $(-n^{2}\leq j<n^{2}-1)$\\

$g'_{n,st}({n^{2}-1\over n})=0$\\

$g'_{n,st}(x)=g'_{n,st}({[nx]\over n})$   $(x\in\mathcal{R}_{n})$\\

and the shift $g^{sh}_{n,st}$ by;\\

$g^{sh}_{n,st}({j\over n})=g_{n,st}({j+1\over n})$  $(-n^{2}\leq j<n^{2}-1)$\\

$g^{sh}_{n,st}({n^{2}-1\over n})=0$\\

$g^{sh}_{n,st}(x)=g^{sh}_{n,st}({[nx]\over n})$   $(x\in\mathcal{R}_{n})$\\

So both are $\mathfrak{C}_{n,st}$-measurable.

\end{defn}

\begin{lemma}{Discrete Calculus Lemmas}\\
\label{disccalc}

Let $\{g_{n,st},h_{n,st}\}$ be $\mathfrak{C}_{n,st}$-measurable and let $\{g'_{n,st},h'_{n,st},g^{sh}_{n,st},h^{sh}_{n,st}\}$ be as in Definition \ref{discretederivshift}. Then;\\

$(i)$. $\int_{\mathcal{R}_{n}}g'_{n,st} d\lambda_{n,st}=g_{n,st}({n^{2}-1\over n})-g_{n,st}(-n)$\\

$(ii)$. $(g_{n,st}h_{n,st})'=g'_{n,st}h^{sh}_{n,st}+g_{n,st}h'_{n,st}$\\

$(iii)$. $\int_{\mathcal{R}_{n}}g'_{n,st}h_{n,st}d\lambda_{n,st}=-\int_{\mathcal{R}_{n}}g^{sh}_{n,st}h'_{n,st}d\lambda_{n,st}+gh_{n,st}({n^{2}-1\over n})-gh_{n,st}(-n)$\\
\end{lemma}
\begin{proof}

$(i)$. We have, using Definition \ref{discretederivshift}, see also Definition \ref{internint};\\

$\int_{\mathcal{R}_{n}}g'_{n,st} d\lambda_{n,st}$\\

$={1\over n}\sum_{j=-n^{2}}^{n^{2}-2}g'_{n,st}({j\over n})$\\

$={1\over n}\sum_{j=-n^{2}}^{n^{2}-2}n(g_{n,st}({j+1\over n})-g_{n,st}({j\over n}))$\\

$=g_{n,st}({n^{2}-1\over n})-g_{n,st}(-n)$\\

$(ii)$. Again, by Definition \ref{discretederivshift};\\

$(gh_{n,st})'({j\over n})$\\

$=n(gh_{n,st}({j+1\over n})-gh_{n,st}({j\over n}))$\\

$=n((g_{n,st}({j+1\over n})-g_{n,st}({j\over n}))h_{n,st}({j+1\over n})+g_{n,st}({j\over n})(h_{n,st}({j+1\over n})-h_{n,st}({j\over n})))$\\

$=g'_{n,st}({j\over n})h^{sh}_{n,st}({j\over n})+g_{n,st}({j\over n})h'_{n,st}({j\over n})$\\

$=(g'_{n,st}h^{sh}_{n,st}+g_{n,st}h'_{n,st})({j\over n})$  $(-n^{2}\leq j\leq n^{2}-2)$\\

$(g'_{n,st}h^{sh}_{n,st}+g_{n,st}h'_{n,st})({n^{2}-1\over n})=(g_{n,st}h_{n,st})'({n^{2}-1\over n})=0$\\

$(iii)$. By $(i),(ii)$;\\

$\int_{\mathcal{R}_{n}}(h_{n,st}g_{n,st})'d\lambda_{n,st}$\\

$=\int_{\mathcal{R}_{n}}(h'_{n,st}g^{sh}_{n,st}+h_{n,st}g'_{n,st}) d\lambda_{n,st}$\\

$=gh_{n,st}({n^{2}-1\over n})-gh_{n,st}(-n)$\\
\end{proof}

\begin{defn}
\label{functiontails}

For $n\in\mathcal{N}$, we let $\theta_{n}:\mathcal{R}\rightarrow\mathcal{C}$ be defined by $\theta_{n}(t)=n(exp({-\pi it\over n})-1)$, and let $\beta_{n}:\mathcal{R}\rightarrow\mathcal{C}$ be defined by $\beta_{n}(t)=n(exp({\pi it\over n})-1)$. We let $\{\phi_{n},\psi_{n}\}$ denote their $\mathfrak{C}_{n}$-measurable counterparts on $\mathcal{R}_{n}$.
If $g_{n,st}$ is $\mathfrak{C}_{n,st}$-measurable, we let;\\

$C_{n}(t)=g_{n,st}({n^{2}-1\over n})exp_{n,st}(-\pi i{n^{2}-1\over n}t)-g_{n,st}(-n)exp_{n,st}(-\pi i(-n)t)$\\

$D_{n}(t)=-{1\over n}g_{n,st}(-n)exp_{n,st}(\pi i{t\over n})exp_{n,st}(-\pi i(-n)t)$.\\

$C'_{n}(t)=-g'_{n,st}(-n)exp_{n,st}(-\pi i(-n)t)$\\

$D'_{n}(t)=-{1\over n}g'_{n,st}(-n)exp_{n,st}(\pi i{t\over n})exp_{n,st}(-\pi i(-n)t)$.\\

$E_{n}(t)=\phi_{n}(t)D_{n}(t)-C_{n}(t)$\\

$E'_{n}(t)=\phi_{n}(t)D'_{n}(t)-C'_{n}(t)$\\

$F_{n}(t)=\psi_{n}(t)\phi_{n}(t)D_{n}(t)-\psi_{n}(t)C_{n}(t)+\phi_{n}(t)D'_{n}(t)-C'_{n}(t)$\\

considered as $\mathfrak{C}_{n,st}$-measurable functions.
\end{defn}

\begin{lemma}{Discrete Fourier transform}\\
\label{dft}

Let $g_{n,st}$ be $\mathfrak{C}_{n,st}$-measurable. Then, for $t\neq 0$;\\

$\hat{g}_{n,st}(t)={\hat{g'}_{n,st}(t)+E_{n}(t)\over\psi_{n}(t)}={\hat{g''}_{n,st}(t)+F_{n}(t)\over\psi_{n}^{2}(t)}$\\

\end{lemma}

\begin{proof}
We have, using Lemma \ref{disccalc}(iii), that;\\

$\hat{g'}_{n,st}(t)=\int_{\mathcal{R}_{n}} g'_{n,st}(x)exp_{n,st}(-\pi ixt)d\lambda_{n,st}(x)$\\

$=-\int_{\mathcal{R}_{n}} g^{sh}_{n,st}(x)exp'_{n,st}(-\pi ixt)d\lambda_{n,st}(x)+ C_{n}(t)$\\

Moreover, for $-n^{2}\leq j<n^{2}-1$;\\

$exp'_{n,st}(-\pi i{j\over n}t)=n(exp_{n,st}(-\pi i{j+1\over n}t)-exp_{n,st}(-\pi i{j\over n}t))$\\

$=nexp_{n,st}(-\pi i{j\over n}t)(exp_{n,st}(-\pi i{t\over n})-1)$\\

$=exp_{n,st}(-\pi i{j\over n}t)\phi_{n}(t)$.\\

Hence, noticing that $g^{sh}_{n,st}({n^{2}-1\over n})=0$, by Definition \ref{discretederivshift};\\

$\hat{g'}_{n,st}(t)=-\int_{\mathcal{R}_{n}} g^{sh}_{n,st}(x)exp_{n,st}(-\pi ixt)\phi_{n}(t)d\lambda_{n,st}(x)+C_{n}(t)$\\

$=-\phi_{n}(t)\hat{g^{sh}_{n,st}}(t)+C_{n}(t)$\\

We also have, using a change of variables, and Definition \ref{discretederivshift}, that;\\

$\hat{g^{sh}_{n,st}}(t)=\int_{\mathcal{R}_{n}}g^{sh}_{n,st}(x)exp_{n,st}(-\pi ixt)d\lambda_{n,st}(x)$\\

$=\int_{1-n^{2}\over n}^{n}g_{n,st}(u)exp_{n,st}(-\pi i(u-{1\over n})t) d\lambda_{n,st}(u)$\\

$=exp_{n,st}(\pi i{t\over n})(\hat{g}_{n,st}(t)-{1\over n}g_{n,st}(-n)exp_{n,st}(-\pi i(-n)t))$\\

$=exp_{n,st}(\pi i{t\over n})\hat{g}_{n,st}(t)+D_{n}(t)$\\

Therefore;\\

$\hat{g'}_{n,st}(t)=-\phi_{n}(t)exp_{n,st}(\pi i{t\over n})\hat{g}_{n,st}(t)-\phi_{n}(t)D_{n}(t)+C_{n}(t)$\\

$=\psi_{n}(t)\hat{g}_{n,st}(t)-E_{n}(t)$\\

and by the same calculation;\\

$\hat{g''}_{n,st}(t)=\psi_{n}(t)\hat{g'}_{n,st}(t)-E'_{n}(t)$\\

$=\psi_{n}(t)(\psi_{n}(t)\hat{g}_{n,st}(t)-E_{n}(t))-E'_{n}(t)$\\

$=\psi_{n}^{2}(t)\hat{g}_{n,st}(t)-F_{n}(t)$\\

Rearranging, we have that, for $t\neq 0$;\\

$\hat{g}_{n,st}(t)={\hat{g'}_{n,st}(t)+E_{n}(t)\over\psi_{n}(t)}={\hat{g''}_{n,st}(t)+F_{n}(t)\over\psi_{n}^{2}(t)}$\\

as required.

\end{proof}

\begin{lemma}
\label{unifbounded}
If $g\in S(\mathcal{R})$, then the functions $\hat{g''}_{n,st}(t)$ and $F_{n}(t)$ are uniformly bounded, independently of $n$, for $n\geq 2$.
\end{lemma}

\begin{proof}
Observing that;\\

 $|D_{n}(t)|\leq {1\over n}|g_{n,st}|(-n)$\\

 $|D'_{n}(t)|\leq {1\over n}|g'_{n,st}|(-n)$\\

 $|\phi_{n}(t)|\leq 2n$, $|\psi_{n}(t)|\leq 2n$\\

$|C_{n}(t)|\leq |g_{n,st}|({n^{2}-1\over n})+|g_{n,st}|(-n)$\\

$|C'_{n}(t)|\leq|g'_{n,st}|(-n)$\\

we obtain;\\

$|F_{n}(t)|\leq 6n|g_{n,st}|(-n)+2n|g_{n,st}|({n^{2}-1\over n})+3|g'_{n,st}|(-n)$\\

$\leq 6n|g_{n,st}|(-n)+2n|g_{n,st}|({n^{2}-1\over n})+3n|g_{n,st}|({1-n^{2}\over n})+3n|g_{n,st}|(-n)$\\

$=9n|g_{n,st}|(-n)+2n|g_{n,st}|({n^{2}-1\over n})+3n|g_{n,st}|({1-n^{2}\over n})$\\

As $g\in S(\mathcal{R})$, there exist a constant $D_{1}$, such that $|g(x)|\leq{D_{1}\over |x|}$, $(x\neq 0)$. Then;\\

$|F_{n}(t)|\leq D_{1}({9n\over n}+5{n^{2}\over n^{2}-1})\leq 16D_{1}$\\

We now calculate;\\

$|\hat{g''}_{n,st}|(t)=|\int_{\mathcal{R}_{n}}g''_{n,st}(x)exp_{n}(-\pi ixt) d\lambda_{n}(x)|$\\

$\leq\int_{\mathcal{R}_{n}}|g''_{n,st}|(x) d\lambda_{n}(x)$\\

$={1\over n}\sum_{j=-n^{2}}^{n^{2}-2}|g''_{n,st}|({j\over n})$\\

$={1\over n}(\sum_{j=-n^{2}}^{n^{2}-3}n|g'_{n,st}({j+1\over n})-g'_{n,st}({j\over n})|)+D_{n}$\\

where $D_{n}=|g'_{n,st}|({n^{2}-2\over n})$.\\

$|\hat{g''}_{n,st}|(t)\leq (\sum_{j=-n^{2}}^{n^{2}-3}|g'_{n,st}({j+1\over n})-g'_{n,st}({j\over n})|)+D_{n}$\\

Without loss of generality, we can assume that $g$ is real valued, otherwise, take real and imaginary parts. Then, by the mean value theorem, for $-n^{2}\leq j\leq n^{2}-3$;\\

$g'_{n,st}({j\over n})=g'({j\over n}+c(j,n))$, where $0<c(j,n)<{1\over n}$\\

$|\hat{g''}_{n,st}|(t)\leq (\sum_{j=-n^{2}}^{n^{2}-3}|g'({j+1\over n}+c(j+1,n))-g'({j\over n}+c(j,n))|)+D_{n}$\\

$=(\sum_{j=-n^{2}}^{n^{2}-3}|\int_{{j\over n}+c(j,n)}^{{j+1\over n}+c(j+1,n)}g''(x) dx|)+D_{n}$ (by the FTC)\\

$\leq(\sum_{j=-n^{2}}^{n^{2}-2}\int_{{j\over n}+c(j,n)}^{{j+1\over n}+c(j+1,n)}|g''|(x)dx)+D_{n}$\\

$=(\int_{-n+c(-n^{2},n)}^{{n^{2}-1\over n}+c(n^{2}-1,n)}|g''|(x)dx)+D_{n}$\\

$\leq (\int_{-n}^{n}|g''(x)|dx)+D_{n}\leq M +2B$\\

where $M=||g''||_{L^{1}(\mathcal{R})}$, and $B=||g||_{C(\mathcal{R})}$.

\end{proof}

\begin{lemma}
\label{tailvanish}
If $g\in S(\mathcal{R})$ and $\epsilon>0$ is standard, there exists a constant $N(\epsilon)\in{\mathcal N}_{>0}$, such that for all $n>N(\epsilon)$, for all $L,L'\in{\mathcal N}$ with $N(\epsilon)<|L|\leq |L'|\leq n$, $LL'>0$;\\

$\int_{L}^{L'}|\hat{g}_{n,st}|(t) d\lambda_{n}(t)<\epsilon$\\

\end{lemma}

\begin{proof}

We first calculate;\\

$|\psi_{n}(t)|=n|exp({\pi it\over n})-1|$\\

$=n|cos({\pi t\over n})+isin({\pi t\over n})-1|$\\

$=n((cos({\pi t\over n})-1)^{2}+sin({\pi t\over n})^{2})^{1\over 2}$\\

$=n((2-2cos({\pi t\over n}))^{1\over 2})$\\

$=\sqrt{2}n(2sin^{2}({\pi t\over 2n}))^{1\over 2}$\\

$=2n|sin({\pi t\over 2n})|\geq 2n({|t|\over n})=2|t|$ $(-n\leq t<n)$\\

$|\psi_{n}(t)|^{2}\geq 4t^{2}$ $(-n\leq t<n)$ $(*)$\\

Letting $W$ denote the bound obtained in Lemma \ref{unifbounded}, using Lemma \ref{dft}, $(*)$, and, assuming, without loss of generality, that $0\leq L\leq L'$;\\

$\int_{L}^{L'}|\hat{g}|_{n,st}(t)d\lambda_{n}(t)$\\

$\leq\int_{L}^{n}|\hat{g}|_{n,st}(t)d\lambda_{n}(t)$\\

$\leq\int_{L}^{n}{W\over 4t^{2}}d\lambda_{n}(t)$\\

$={1\over n}\sum_{j=Ln}^{n^{2}-1}{W\over 4({j\over n})^{2}}$\\

$=n\sum_{j=Ln}^{n^{2}-1}{W\over 4j^{2}}$\\

$\leq n\int_{Ln-1}^{n^{2}-1}{W\over 4x^{2}} dx$\\

$=n[{-W\over 4x}]^{n^{2}-1}_{Ln-1}={Wn\over 4}({1\over Ln-1}-{1\over n^{2}-1})\leq {W\over 4}({1\over L-1}+{1\over n-1})<\epsilon$\\

if $min(n,L)>N(\epsilon)={W\over 2\epsilon}+1$\\

\end{proof}

We can now show the analogous result to Theorem \ref{gSinteg};\\

\begin{theorem}
\label{transgSinteg}

Let $g\in S({\mathcal R})$, then $\hat{g}_{\eta}$, as given in Definition \ref{transfdefs}, is $S$-integrable on $\overline{{\mathcal R}_{\eta}}$. Moreover ${^{\circ}\hat{g}}_{\eta}=st^{*}(\hat{g}_{\infty})$, almost everywhere $L(\lambda_{\eta})$, and;\\

${^{\circ}\int_{\overline{{\mathcal R}_{\eta}}}\hat{g}_{\eta} d\lambda_{\eta}}=\int_{\overline{{\mathcal R}_{\eta}}}st^{*}(\hat{g}_{\infty}) d L(\lambda_{\eta})=\int_{\mathcal{R}} \hat{g} d\mu$\\
\end{theorem}
\begin{proof}
By Lemma \ref{tailvanish};\\

$\mathcal{R}\models (\forall n_{(n>N(\epsilon))})(\forall L,N_{(LN\geq 0, N(\epsilon)<|L|,|N|<n)})\int_{L}^{N}|\hat{g}_{n,st}| d\lambda_{n,st} <\epsilon$\\

Hence, the corresponding statement is true in ${^{*}\mathcal{R}}$. In particular, if $\eta$ is infinite, and $\{L,N\}$ are infinite, of the same sign, belonging to $\overline{\mathcal{R}}_{\eta}$, we have that;\\

$\int_{L}^{N}|\hat{g}_{\eta}| d\lambda_{\eta} < \epsilon$\\

As $\epsilon$ was arbitrary we conclude that;\\

$\int_{L}^{N}|\hat{g}_{\eta}| d\lambda_{\eta}\simeq 0$ $(*)$\\

for all infinite $\{L,N\}$, of the same sign, in $\overline{\mathcal{R}}_{\eta}$. Now, using Definition \ref{transfdefs} and the fact that $|exp_{\eta}(-\pi ixt)|\leq 1$, by transfer, we have, for $t\in\overline{\mathcal{R}}_{\eta}$;\\

$|\hat{g}_{\eta}(t)|\leq\int_{\overline{{\mathcal R}_{\eta}}}|g_{\eta}(x)d\lambda_{\eta}=C$\\

where $C$ is finite, as, by Theorem \ref{gSinteg}, $g_{\eta}$ is $S$-integrable. It follows that for $n\in\mathcal{N}$, the functions $\hat{g}_{\eta}\chi_{[-n,n]}$ are finite, in the sense of Definition \ref{finite}. Now, proceeding as in Theorem \ref{gSinteg}, we obtain that $\hat{g}_{\eta}$ is $S$-integrable. If $t\in{\mathcal{R}_{\eta}}$, the function $r_{t}(x)=g_{\eta}(x)exp_{\eta}(-\pi ixt)$ is $S$-integrable, by Theorem \ref{andcrit}(i), as $|r_{t}|\leq |g_{\eta}|$, and $g_{\eta}$ is $S$-integrable, by Theorem \ref{gSinteg}. Then, if $t$ is finite, we have;\\

${^{\circ}\hat{g}_{\eta}}(t)={^{\circ}\int_{\overline{{\mathcal R}_{\eta}}}g_{\eta}(x)exp_{\eta}(-\pi ixt)d\lambda_{\eta}(x)}$\\

$=\int_{\overline{{\mathcal R}_{\eta}}}{^{\circ}g_{\eta}}(x){^{\circ}exp_{\eta}}(-\pi ixt)dL(\lambda_{\eta})(x)$\\

$=\int_{x finite}st^{*}(g_{\infty})(x)exp_{\eta}(-\pi i{^{\circ} x}{^{\circ}t})dL(\lambda_{\eta})(x)$\\

$=\int_{x finite}st^{*}(g_{\infty}exp_{-\pi i{^{\circ}t}})(x)dL(\lambda_{\eta})(x)$\\

$=\int_{\mathcal R}g(x)exp(-\pi i{^{\circ}t}x)d\mu(x)=\hat{g}({^{\circ}t})=st^{*}(\hat{g}_{\infty})(t)$ $(**)$\\

using Definition \ref{transfdefs}, Theorem \ref{sintegspec}, Theorem \ref{gSinteg}, continuity of $exp$, see Theorem \ref{nscont}, and Theorem \ref{mmp}. Now suppose there exists $B\in L(\mathfrak{C}_{\eta})$, with $L(\lambda_{\eta})(B)>0$, such that ${^{\circ}\hat{g}_{\eta}}\neq st^{*}(\hat{g}_{\infty})$ on $B$. Then, by $(**)$, we can assume that $B\subset st^{-1}(\{-\infty,+\infty\})$, and $|{^{\circ}\hat{g}_{\eta}}|>0$ on $B$. We can, therefore, suppose that there exists a standard $n\in{\mathcal N}_{>0}$, with $|{^{\circ}\hat{g}_{\eta}}|>{1\over n}$, on $B$. Then for all finite $t'$, using \cite{cut}(Theorem 1.32);\\

${^{\circ}\int_{|t|>t'}|\hat{g}_{\eta}|(t)d\lambda_{\eta}(t)}$\\

$\geq\int_{|t|>t'}|{^{\circ}\hat{g}_{\eta}}|(t)dL(\lambda_{\eta})(t)>{1\over n}L(\lambda_{\eta})(B)$\\

By the Overflow principle, see \cite{cut}, we can find an infinite $L$ such that;\\

$\int_{|t|>L}|\hat{g}_{\eta}|(t)d\lambda_{\eta})(t)>{1\over 2n}L(\lambda_{\eta})(B)$\\

This contradicts $(*)$. Hence ${^{\circ}\hat{g}_{\eta}}=st^{*}(\hat{g}_{\infty})$ a.e $L(\lambda_{\eta})$. The rest of the proof is the same as Theorem \ref{gSinteg}.
\end{proof}

Finally, we have;\\

\begin{theorem}
\label{nspfit}
For $g\in S(\mathcal{R})$, the Fourier Inversion Theorem holds and admits a non standard proof.

\end{theorem}
\begin{proof}
By Lemma \ref{nsit}, we have that;\\

$g_{\eta}(x)={1\over 2}\int_{\overline{\mathcal R}_{\eta}}\hat{g}_{\eta}(t)exp_{\eta}(\pi ixt)d\lambda_{\eta}(t)$ $(*)$\\

for $x\in\overline{{\mathcal R}_{\eta}}$. As in Theorem \ref{transgSinteg}, the function $s_{x}(t)=\hat{g}_{\eta}(t)exp_{\eta}(\pi ixt)$ is $S$-integrable, because, by the same theorem, $\hat{g}_{\eta}$ is $S$-integrable. We now argue as before, and use the result that ${^{\circ}{g}}_{\eta}=st^{*}(\hat{g}_{\infty})$, a.e $L(\lambda_{\eta})$. We have, if $x$ is standard, taking standard parts in $(*)$;\\

$g(x)={^{\circ}g_{\eta}}(x)={1\over 2}\int_{\overline{\mathcal R}_{\eta}}{^{\circ}\hat{g}_{\eta}}(t){^{\circ}exp_{\eta}}(\pi ixt)d L(\lambda_{\eta})(t)$\\

$={1\over 2}\int_{t finite}st^{*}(\hat{g}_{\infty})(t)exp_{\eta}(\pi ix{^{\circ}t})d L(\lambda_{\eta})(t)$\\

$={1\over 2}\int_{t finite}st^{*}(\hat{g}_{\infty}exp_{\pi ix})(t)d L(\lambda_{\eta})(t)$\\

$={1\over 2}\int_{\mathcal{R}}\hat{g}(t)exp(\pi ixt)d\mu(t)$\\

as required.

\end{proof}

\begin{section}{Appendix}

We collect, here, some results in standard and nonstandard analysis which are required in the main proof.\\

\begin{lemma}
\label{extension}

If $(X,\mathfrak{M},\mu)\subset(X,\mathfrak{M'},\mu')$ as standard measure spaces. Then, if $g$ is $\mathfrak{M}$-measurable, and integrable with respect to $(X,\mathfrak{M},\mu)$, then $g$ is integrable with respect to $(X,\mathfrak{M'},\mu')$, and;\\

$\int_{B}g d\mu=\int_{B}g d\mu'$\\

for any $B\in\mathfrak{M}$.
\end{lemma}

\begin{proof}
 First check the result for $\mathfrak{M}$-measurable simple functions, $(*)$. Then, without loss of generality, assume $g\geq 0$. $g$ can be written as in increasing limit of simple $\mathfrak{M}$-measurable functions, see \cite{Rud}( Theorem 1.17). Now apply the Monotone Convergence Theorem, see \cite{Rud}(Theorem 1.26), and $(*)$, to obtain the result. \end{proof}

 \begin{lemma}{Change of Variables}
 \label{cov}
 If $\tau:(X_{1},\mathfrak{C}_{1},\mu_{1})\rightarrow (X_{2},\mathfrak{C}_{2},\mu_{2})$ is measurable and measure preserving, so $\mu_{2}=\tau_{*}\mu_{1}$, then a function $\theta\in L^{1}(X_{2},\mathfrak{C}_{2},\mu_{2})$ iff $\tau^{*}\theta\in L^{1}(X_{1},\mathfrak{C}_{1},\mu_{1})$ and then;\\

  $\int_{C} \theta d\tau_{*}\mu_{1}=\int_{\tau^{-1}(C)} \tau^{*}\theta d\mu_{1}$\\

for $C\in\mathfrak{C}_{2}$.

\end{lemma}

\begin{proof}
This is a simple exercise, using the abstract definition of integration on measure spaces, see \cite{Rud}.
\end{proof}

\begin{defn}
\label{internint}
If $X$ is a hyperfinite interval, that is $X=\{x\in{^{*}\mathcal{R}}:{a\over\eta}\leq x<{b\over\eta}\}$, where $a,b\in{^{*}\mathcal{Z}}$, $\mathfrak{A}$ is the set of all internal unions of intervals of the form $[{j\over\eta},{j+1\over\eta})$, where $a\leq j<b$, $j\in{^{*}\mathcal{Z}}$ and $\nu$ is the counting measure given by $\nu([{j\over\eta},{j+1\over\eta}))={1\over \eta}$, then the internal integral takes the form;\\

$\int_{X} f d\nu={1\over\eta}\sum_{j=a}^{b-1}f({j\over\eta})$\\

 Some authors, see \cite{cut}, prefer to use a discrete version of the hyperfinite interval, in which $X=\{{j\over\eta}:a\leq j<b, j\in{^{*}\mathcal{Z}}\}$, $\mathfrak{A}$ is the set of internal subsets, and $\nu$ is the counting measure given by $\nu(x)={1\over\eta}$, for $x\in X$. Of course the two interpretations are equivalent and the internal integral takes the same form.\\

\end{defn}

\begin{theorem}
\label{stlim}
Let $(s_{n})_{n\in{\mathcal N}}$ be a standard infinite sequence, then the following are equivalent;\\

$(i)$. $lim_{n\rightarrow\infty}s_{n}=s$.\\

$(ii)$. $s_{n}\simeq s$ for all infinite $n$.\\

\end{theorem}

\begin{proof}
See \cite{Rob}.
\end{proof}

The following is a slight generalisation of Theorem 3.5.13 of \cite{Rob};\\

\begin{theorem}
\label{intlim}
Let $(s_{n})_{n\in{^{*}{\mathcal N}}}$ be an internal sequence, enumerated by an internal $g$, not necessarily the transfer of a standard one. Suppose there exists an infinite $\omega\in{^{*}{\mathcal N}}$ with $s_{\omega'}\simeq 0$, for all infinite $\omega'$ with $\omega'<\omega$, then;\\

$lim_{n\rightarrow\infty}(^{\circ}s_{n})=0$ in the standard sense.\\

\end{theorem}

\begin{proof}
Let $\epsilon>0$ be standard, then $|g(\omega')|<\epsilon$ for all infinite $\omega'$ with $\omega'<\omega$. Let;\\

$A=\{m\in{^{*}{\mathcal N}}:|g(n)|<\epsilon, if\ m\leq n<\omega\}$\\

Then $A$ is internal and contains arbitrarily small positive infinite numbers. By the underflow principle, see \cite{cut}, it contains a positive finite number $m_{0}\in{\mathcal N}$. In particular, $|g(n)|<\epsilon$, for all $n\in{\mathcal N}$, with $n\geq m_{0}$. It follows that $|^{\circ}(g(n))|\leq\epsilon$, for all $n\in{\mathcal N}$, with $n\geq m_{0}$. Hence, as $\epsilon$ was arbitrary, the result follows.
\end{proof}

\begin{theorem}
\label{nscont}

Let $f:{\mathcal R}\rightarrow{\mathcal R}$ be a standard function, and $b\in{\mathcal R}$, then the following are equivalent;\\

(i). $f$ is continuous at $b$.\\

(ii). ${^{*}f}(x)\simeq {^{*}f}(b)$ for all $x\in\mu(b)$.\\

 where $\mu(b)$ is the monad of $b$. (\footnote{\label{nscontfin} This is often applied in the following form; if $x\simeq y$ belong to ${^{*}{\mathcal R}}_{fin}$, then ${^{*}f}(x)\simeq{^{*}f}(y)$. This follows from the facts that $x$ and $y$ have a standard part in ${\mathcal R}$, and $\simeq$ is an equivalence relation.}).\\

\end{theorem}

\begin{defn}{Anderson}\\
\label{finite}

 Let $(X,\mathfrak{A},\nu)$ be an internal measure space, in the sense of \cite{Loeb}, and let $f:X\rightarrow {^{*}{\mathcal R}}$ be ${\mathfrak A}$-measurable. Then we say $f$ is finite if;\\

 (i). There exists an $n\in{\mathcal N}$, with $|f(x)|<n$, for all $x\in X$.\\

 (ii). $f$ is supported on a set $A$ with $\nu(A)$ finite.\\

 \end{defn}

 We have;\\

 \begin{theorem}{Anderson's Criteria}\\
\label{andcrit}

 Let $(X,\mathfrak{A},\nu)$ be as in Definition \ref{finite} and let $f:X\rightarrow^{*}{\mathcal R}$ be ${\mathfrak A}$-measurable.\\

 $(i)$. If $F$ is $S$-integrable, with $|f|\leq F$, then $f$ is $S$-integrable.\\

 $(ii)$. If ${\mathfrak A}$ is a $^{*}{\sigma}$-algebra, then $f$ is $S$-integrable iff there exists a sequence of finite functions $(f_{n})_{n\in\mathcal{N}}$ such that;\\

 $^{\circ}(\int_{X}|f-f_{n}| d\nu)\rightarrow 0$ as $n\rightarrow\infty$.\\

 \end{theorem}
 \begin{proof}
 See \cite{and}.

 \end{proof}

 \begin{theorem}
 \label{sintegspec}
Let $(X,\mathfrak{A},\nu)$ be as in Definition \ref{finite}, let $(X,L(\mathfrak{A}),L(\nu))$ be the corresponding Loeb space and let $f:X\rightarrow^{*}{\mathcal R}$ be ${\mathfrak A}$-measurable. Then the following are equivalent;\\

 $(i)$. $f$ is $S$-integrable.\\

 $(ii)$. ${^{\circ}f}$ is integrable with respect to $L(\nu)$, and;\\

 $^{\circ}\int_{A} f d\nu = \int_{A}{^{\circ}f} d L(\nu)$ for any $A\in{\mathfrak A}$.\\

\end{theorem}

\begin{proof}
See \cite{and} or Theorem 3.24 of \cite{deP}.
\end{proof}

\end{section}

\end{document}